\documentclass[oneside,leqno,11pt]{article}
\usepackage{amssymb,amsmath,latexsym, stmaryrd,amsthm}

\def\ga{\mathfrak{a}}

\def\gg{\mathfrak{g}}
\def\ggl{\mathfrak{gl}}

\def\gk{\mathfrak{k}}
\def\gl{\mathfrak{l}}
\def\gm{\mathfrak{m}}
\def\gn{\mathfrak{n}}

\def\gp{\mathfrak{p}}
\def\gq{\mathfrak{q}}
\def\gr{\mathfrak{r}}
\def\gs{\mathfrak{s}}
\def\gsl{\mathfrak{sl}}
\def\gso{\mathfrak{so}}
\def\gsu{\mathfrak{su}}
\def\gsp{\mathfrak{sp}}
\def\gt{\mathfrak{t}}
\def\gu{\mathfrak{u}}

\def\gz{\mathfrak{z}}
\def\gS{\mathfrak{S}}

\def\Ad{{\rm Ad}}

\def\trace{{\rm trace\,\,}}

\def\Ind{{\rm Ind\,}}

\def\C{\mathbb{C}}
\def\D{\mathbb{D}}
\def\E{\mathbb{E}}
\def\F{\mathbb{F}}

\def\H{\mathbb{H}}

\def\R{\mathbb{R}}

\def\Z{\mathbb{Z}}

\def\cC{\mathcal{C}}

\def\cF{\mathcal{F}}

\def\cH{\mathcal{H}}

\def\cU{\mathcal{U}}

\renewcommand{\thesection}{\arabic{section}}

\newtheorem{theorem}[equation]{Theorem}

\newtheorem{lemma}[equation]{Lemma}

\newtheorem{proposition}[equation]{Proposition}

\newtheorem{definition}[equation]{Definition}

\newtheorem{example}[equation]{Example}

\title{Principal series representations of infinite dimensional Lie groups, I:
Minimal parabolic subgroups}

\author{Joseph A. Wolf\footnote{Research partially supported by a Simons Foundation grant
\endgraf
AMS 2010 Subject Class: Primary 32L25; Secondary 22E46, 32L10}}

\date{}

\begin{document}
\maketitle

\abstract{We study the structure of minimal parabolic subgroups of the classical
infinite dimensional real simple Lie groups, corresponding to the
classical simple direct limit Lie algebras.  This depends on the recently
developed structure of parabolic subgroups and subalgebras
that are not necessarily
direct limits of finite dimensional parabolics.  We then 
discuss the use of that structure theory for the infinite dimensional
analog of the classical principal series representations.
We look at the unitary representation theory of the
classical lim--compact groups $U(\infty)$, $SO(\infty)$ and $Sp(\infty)$
in order to construct the inducing representations, and we indicate
some of the analytic considerations in the actual construction of the induced
representations.}

\section{Introduction}\label{sec0}
\setcounter{equation}{0}

This paper reports on some recent developments in a program of 
extending aspects of real semisimple group representation theory to 
infinite dimensional real Lie groups.  The
finite dimensional theory is entwined with the structure of parabolic
subgroups, and that structure has recently been worked out for the classical
direct limit groups such as $SL(\infty,\R)$ and $Sp(\infty;\R)$.  
Here we explore the consequences of that structure theory for the construction
of the counterpart of various Harish--Chandra series of representations,
specifically the principal series.
\smallskip

The representation theory of finite dimensional real semisimple Lie groups 
is based on the now--classical constructions and Plancherel Formula of 
Harish--Chandra.  Let $G$ be a real
semisimple Lie group, e.g. $SL(n;\R)$, $SU(p,q)$, $SO(p,q)$, \dots.  Then
one associates a series of representations to each conjugacy class of
Cartan subgroups.  Roughly speaking this goes as follows.
Let $Car(G)$ denote the set of conjugacy classes $[H]$ of Cartan subgroups 
$H$ of $G$.  Choose $[H] \in Car(G)$, $H \in [H]$, and an irreducible unitary 
representation $\chi$ of $H$.  Then we have a ``cuspidal'' parabolic subgroup 
$P$ of $G$ constructed from $H$, and a unitary representation 
$\pi_\chi$ of $G$ constructed from $\chi$ and $P$.  Let
$\Theta_{\pi_\chi}$ denote the distribution character of $\pi_\chi$\, .
The Plancherel Formula: if $f\in \cC(G)$, the Harish-Chandra Schwartz space,
then
\addtocounter{equation}{1}
\begin{equation}\label{planch}\tag{\thesection.\arabic{equation}}
f(x) = \sum_{[H] \in Car(G)} \,\,\int_{\widehat{H}}
        \Theta_{\pi_\chi}(r_xf) d\mu_{[H]}(\chi)
\end{equation}
where $r_x$ is right translation and $\mu_{[H]}$ is Plancherel measure
on the unitary dual $\widehat{H}$.
\smallskip

In order to consider any elements of this theory in the context of
real semisimple direct
limit groups, we have to look more closely at the construction of the
Harish--Chandra series that enter into (\ref{planch}).
\smallskip

Let $H$ be a Cartan subgroup of $G$.  It is stable under a
Cartan involution $\theta$, an involutive automorphism of $G$ whose
fixed point set $K = G^\theta$ is a maximal compactly embedded\footnote{A
subgroup of $G$ is {\sl compactly embedded} if it has compact image under
the adjoint representation of $G$.} subgroup.
Then $H$ has a $\theta$--stable decomposition
$T \times A$ where $T = H \cap K$ is the compactly embedded part 
and (using lower case Gothic letters for Lie algebras)
$\exp : \ga \to A$ is a bijection.  Then $\ga$ is commutative and acts
diagonalizably on $\gg$.  Any choice of positive $\ga$--root
system defines a parabolic subalgebra $\gp = \gm + \ga + \gn$ in $\gg$ 
and thus defines a parabolic subgroup $P = MAN$ in $G$.  If $\tau$ is an
irreducible unitary representation of $M$ and $\sigma \in \ga^*$ then
$\eta_{\tau,\sigma}: man \mapsto e^{i\sigma(\log a)}\tau(m)$ is a well
defined irreducible unitary representation of $P$.  The equivalence class
of the unitarily induced representation $\pi_{\tau,\sigma} =
\Ind_P^G(\eta_{\tau,\sigma})$ is independent of the choice of positive 
$\ga$--root system.  The group $M$ has (relative) discrete series 
representations,
and $\{\pi_{\tau,\sigma} \mid \tau \text{ is a discrete series rep of } M\}$
is the series of unitary representations associated to 
$\{\Ad(g)H \mid g \in G\}$.
\smallskip

One of the most difficult points here is dealing with the discrete series.
In fact the possibilities of direct limit representations of direct limit
groups are somewhat limited except in cases where one can pass cohomologies
through direct limits without change of cohomology degree.  See \cite{N}
for limits of holomorphic discrete series, \cite{NRW} for Bott--Borel--Weil
theory in the direct limit context, \cite{H} for some nonholomorphic
discrete series cases, and \cite{W5} for principal series of
classical type.  The principal series representations in (\ref{planch}) are
those for which $M$ is compactly embedded in $G$, equivalently the ones
for which $P$ is a minimal parabolic subgroup of $G$.  
\smallskip

Here we work out the structure of the minimal parabolic subgroups of
the finitary simple real Lie groups and discuss construction of
the associated principal
series representations.  As in the finite dimensional case, a minimal
parabolic has structure $P = MAN$.  Here $M = P \cap K$ is a (possibly infinite)
direct product of torus groups, compact classical groups such as $Spin(n)$,
$SU(n)$, $U(n)$ and $Sp(n)$, and their classical direct limits $Spin(\infty)$,
$SU(\infty)$, $U(\infty)$ and $Sp(\infty)$ (modulo intersections and discrete 
central subgroups).
\smallskip

Since this setting is not standard we must start by sketching the background.
In Section \ref{sec1} we recall the classical simple real direct limit 
Lie algebras and Lie groups.  There are no surprises.  Section \ref{sec2} 
sketches their relatively recent theory of complex parabolic subalgebras.  
It is a little bit complicated and there are some surprises.  Section 
\ref{sec3} carries those results over to real parabolic subalgebras.  There
are no new surprises.  Then in Sections \ref{sec4} and \ref{sec5} we deal 
with Levi components and Chevalley decompositions.  That completes the background.
\smallskip

In Section \ref{sec6} we examine the structure real group structure of 
Levi components of real parabolics.  Then we specialize this to minimal
self--normalizing 
parabolics in Section \ref{sec7}.  There the Levi components are locally
isomorphic to direct sums in an explicit way of subgroups that are either
the compact classical groups $SU(n)$, $SO(n)$ or $Sp(n)$, or
their limits $SU(\infty)$, $SO(\infty)$ or $Sp(\infty)$.  The Chevalley
(maximal reductive part) components are slightly more complicated, for
example involving extensions $1 \to SU(*) \to U(*) \to T^1 \to 1$ as
well as direct products with tori and vector groups.  The main result is
Theorem \ref{lang-alg}, which gives the structure of the minimal
self--normalizing parabolics in terms similar to those of the finite
dimensional case.  Proposition \ref{construct-p} then gives an explicit
construction for minimal parabolics with a given Levi factor.

In Section \ref{sec8} we discuss the various possibilities for the 
inducing representation.  There are many
good choices, for example tame representations or
more generally representations that are factors of type $II$.
The theory is at such an early stage that the best choice is not yet clear.

Finally, in Section \ref{sec9} we indicate construction of 
the induced representations in our infinite dimensional setting.  Smoothness
conditions do not introduce surprises, but unitarity is a problem, and
we defer details of that construction to \cite{W7} and applications to 
\cite{W8}. 
\smallskip

I thank Elizabeth Dan--Cohen and Ivan Penkov for many very helpful discussions
on parabolic subalgebras and Levi components.

\section{The Classical Simple Real Groups}\label{sec1}
\setcounter{equation}{0}

In this section we recall the real simple countably
infinite dimensional
locally finite (``finitary'') Lie algebras and the corresponding Lie groups.  
This material follows from results in \cite{B}, \cite{BS} and \cite{DCPW}.
\smallskip

We start with the three classical simple locally finite countable--dimensional 
Lie algebras $\gg_\C = \varinjlim \gg_{n,\C}$, and their real forms
$\gg_\R$.  The Lie algebras $\gg_\C$ are the classical direct limits,
\addtocounter{equation}{1}
\begin{equation}\label{cpx-lie}\tag{\thesection.\arabic{equation}}
\begin{aligned}
& \gsl(\infty,\C) = \varinjlim \gsl(n;\C),\\
& \gso(\infty,\C) = \varinjlim \gso(2n;\C) = \varinjlim \gso(2n+1;\C), \\
& \gsp(\infty,\C) = \varinjlim \gsp(n;\C),
\end{aligned}
\end{equation}
where the direct systems are given by the inclusions of the form
$A \mapsto (\begin{smallmatrix} A & 0 \\ 0 & 0 \end{smallmatrix} )$.
We will also consider the locally reductive algebra
$\ggl(\infty;\C) = \varinjlim \ggl(n;\C)$ along with $\gsl(\infty;\C)$.
The direct limit process of (\ref{cpx-lie}) defines the universal enveloping
algebras
\addtocounter{equation}{1}
\begin{equation}\label{univ-cpx-lie}\tag{\thesection.\arabic{equation}}
\begin{aligned}
& \cU(\gsl(\infty,\C)) = \varinjlim \cU(\gsl(n;\C)) \text{ and }
	\cU(\ggl(\infty,\C)) = \varinjlim \cU(\ggl(n;\C)), \\
& \cU(\gso(\infty,\C)) = \varinjlim \cU(\gso(2n;\C)) = \varinjlim \cU(\gso(2n+1;\C)),        
        \text{ and }\\
& \cU(\gsp(\infty,\C)) = \varinjlim \cU(\gsp(n;\C)),
\end{aligned}
\end{equation}
\smallskip

Of course each of these Lie algebras $\gg_\C$ has the underlying structure of
a real Lie algebra.  Besides that, their real forms are as follows (\cite{B},
\cite{BS}, \cite{DCPW}).
\smallskip

If  $\gg_\C = \gsl(\infty;\C)$,  then $\gg_\R$ is one of 
$\gsl(\infty;\R) = \varinjlim \gsl(n;\R)$, the real
special linear Lie algebra;
$\gsl(\infty;\H) = \varinjlim \gsl(n;\H)$, the quaternionic
special linear Lie algebra, given by
$\gsl(n;\H) := \ggl(n;\H) \cap \gsl(2n;\C)$;
$\gs\gu(p,\infty) = \varinjlim \gs\gu(p,n)$, the complex special
unitary Lie algebra of real rank $p$; or
$\gs\gu(\infty,\infty) = \varinjlim \gs\gu(p,q)$, complex
special unitary algebra of infinite real rank.
\smallskip

If  $\gg_\C = \gso(\infty;\C)$,  then $\gg_\R$ is one of 
$\gso(p,\infty) = \varinjlim \gso(p,n)$, the real orthogonal
Lie algebra of finite real rank $p$;
$\gso(\infty,\infty) = \varinjlim \gso(p,q)$, the real
orthogonal Lie algebra of infinite real rank; or
$\gso^*(2\infty) = \varinjlim \gso^*(2n)$
\smallskip

If  $\gg_\C = \gsp(\infty;\C)$,  then $\gg_\R$ is one of 
$\gsp(\infty;\R) = \varinjlim \gsp(n;\R)$, the real
symplectic Lie algebra;
$\gsp(p,\infty) = \varinjlim \gsp(p,n)$, the quaternionic unitary
Lie algebra of real rank $p$; or
$\gsp(\infty,\infty) = \varinjlim \gsp(p,q)$, quaternionic
unitary Lie algebra of infinite real rank.
\smallskip

If  $\gg_\C = \ggl(\infty;\C)$,  then $\gg_\R$ is one 
$\ggl(\infty;\R) = \varinjlim \ggl(n;\R)$, the real
general linear Lie algebra;
$\ggl(\infty;\H) = \varinjlim \ggl(n;\H)$, the quaternionic
general linear Lie algebra;
$\gu(p,\infty) = \varinjlim \gu(p,n)$, the complex unitary
Lie algebra of finite real rank $p$; or
$\gu(\infty,\infty) = \varinjlim \gu(p,q)$, the complex
unitary Lie algebra of infinite real rank.
\smallskip

As in (\ref{univ-cpx-lie}), given one of these Lie algebras 
$\gg_\R = \varinjlim \gg_{n,\R}$ we have the universal enveloping algebra.
We will need it for the induced representation process.  As in the
finite dimensional case, we use the universal enveloping algebra of the
complexification.  Thus when we write $\cU(\gg_\R)$ it is understood that 
we mean $\cU(\gg_\C)$.  The reason for this is that we will want our
representations of real Lie groups to be representations on complex
vector spaces.
\smallskip

The corresponding Lie groups are exactly what one expects.  First the
complex groups, viewed either as complex groups or as real groups,
\addtocounter{equation}{1}
\begin{equation}\label{cpl_gps}\tag{\thesection.\arabic{equation}}
\begin{aligned}
& SL(\infty;\C) = \varinjlim SL(n;\C) \text{ and } GL(\infty;\C) = 
     \varinjlim GL(n;\C), \\
& SO(\infty;\C) = \varinjlim SO(n;\C) = \varinjlim SO(2n;\C)
     = \varinjlim SO(2n+1;\C), \\
& Sp(\infty;\C) = \varinjlim Sp(n;\C).
\end{aligned}
\end{equation}
The real forms of the complex special and general linear groups 
$SL(\infty;\C)$ and $GL(\infty;\C)$ are
\addtocounter{equation}{1}
\begin{equation}\label{typeA}\tag{\thesection.\arabic{equation}}
\begin{aligned}
& SL(\infty;\R) \text{ and } GL(\infty;\R): 
     \text{ real special/general linear groups, } \\
& SL(\infty;\H): \text{ quaternionic special linear group, } \\
& (S)U(p,\infty): \text{ (special) unitary groups
     of real rank } p < \infty, \\
& (S)U(\infty,\infty): \text{ (special) 
     unitary groups of infinite real rank.}
\end{aligned}
\end{equation}
The real forms of the complex orthogonal and spin groups 
$SO(\infty;\C)$ and $Spin(\infty;\C)$ are
\addtocounter{equation}{1}
\begin{equation}\label{typeBD}\tag{\thesection.\arabic{equation}}
\begin{aligned}
& SO(p,\infty) \text{, } Spin(p;\infty): \text{ real orth./spin 
     groups of real rank } p < \infty,  \\
& SO(\infty,\infty) \text{, } Spin(\infty,\infty): \text{ real
     orthog./spin groups of real rank } \infty, \\
& SO^*(2\infty) = \varinjlim SO^*(2n), \text{ which doesn't have a 
	standard name}
\end{aligned}
\end{equation}
Here 
$$
SO^*(2n) = \{g \in SL(n;\H) \mid g \text{ preserves the form } 
              \kappa(x,y) := \sum x^\ell i \bar y^\ell = {^tx} i \bar y\}.
$$
Alternatively, $SO^*(2n) = SO(2n;\C) \cap U(n,n)$ with  
$$
SO(2n;C) \text{ defined by }
   (u,v) = \sum (u_j v_{n+jr} + v_{n+j}w_j).
$$
Finally, the real forms of the complex symplectic group $Sp(\infty;\C)$ are
\addtocounter{equation}{1}
\begin{equation}\label{typeC}\tag{\thesection.\arabic{equation}}
\begin{aligned}
& Sp(\infty;\R): \text{ real symplectic group,} \\
& Sp(p,\infty): \text{ quaternion unitary group of real rank } p < \infty, \text{ and }\\
& Sp(\infty,\infty): \text{ quaternion unitary group of infinite real rank.}
\end{aligned}
\end{equation}

\section{Complex Parabolic Subalgebras}\label{sec2}
\setcounter{equation}{0}
In this section we recall the structure of parabolic subalgebras of
$\ggl(\infty;\C)$, $\gsl(\infty);\C)$, $\gso(\infty;\C)$ and $\gsp(\infty;\C)$.
We follow Dan--Cohen and Penkov (\cite{DC1}, \cite{DC2}).
\smallskip

We first describe $\gg_\C$ in terms of linear spaces.  Let $V$ and $W$ be 
nondegenerately paired countably infinite dimensional complex vector spaces.   
Then
$\ggl(\infty,\C) = \ggl(V,W) := V\otimes W$ consists of all finite linear
combinations of the rank $1$ operators 
$v\otimes w: x \mapsto \langle w, x\rangle v$.  In
the usual ordered basis of $V = \C^\infty$, parameterized by the positive 
integers, and with the dual basis of $W = V^* = (\C^\infty)^*$, we can
view  $\ggl(\infty,\C)$ can be viewed as infinite matrices with only
finitely many nonzero entries.  However $V$ has more exotic ordered bases, for
example parameterized by the rational numbers, where the matrix picture is
not intuitive.
\smallskip

The rank 1 operator $v\otimes w$ has a well defined trace, so trace is
well defined on $\ggl(\infty,\C)$.  Then $\gsl(\infty,\C)$ is the
traceless part, $\{g \in \ggl(\infty;\C) \mid \trace g = 0\}$.  
\smallskip

In the orthogonal case we can take $V = W$ using the symmetric bilinear
form that defines $\gso(\infty;\C)$.  Then
$$
\gso(\infty;\C) = \gso(V,V) = \Lambda\ggl(\infty;\C) \text{ where }
	\Lambda(v\otimes v') = v\otimes v' - v'\otimes v.
$$
In other words, in an ordered orthonormal basis of $V = \C^\infty$ 
parameterized by the positive integers, $\gso(\infty;\C)$ can be viewed
as the infinite antisymmetric matrices with only finitely many nonzero entries.
\smallskip

Similarly, in the symplectic case we can take $V = W$ using the antisymmetric
bilinear form that defines $\gsp(\infty;\C)$, and then
$$
\gsp(\infty;\C) = \gsp(V,V) = S\ggl(\infty;\C) \text{ where }
        S(v\otimes v') = v\otimes v' + v'\otimes v.
$$
In an appropriate ordered basis of $V = \C^\infty$ parameterized
by the positive integers, $\gsp(\infty;\C)$ can be viewed
as the infinite symmetric matrices with only finitely many nonzero entries.
\smallskip

In the finite dimensional setting, Borel subalgebra means a maximal solvable
subalgebra, and parabolic subalgebra means one that contains a Borel.  It is
the same here except that one must use {\em locally solvable} to avoid the
prospect of an infinite derived series.

\begin{definition}
{\rm A {\em Borel subalgebra} of $\gg_\C$ is a maximal 
locally solvable subalgebra.  A {\em parabolic subalgebra} of $\gg_\C$ 
is a subalgebra that contains a Borel subalgebra.}
\hfill $\diamondsuit$
\end{definition}
\smallskip

In the finite dimensional setting a parabolic subalgebra is the stabilizer
of an appropriate nested sequence of subspaces (possibly with an orientation
condition in the orthogonal group case).  In the infinite dimensional setting
here,  one must be very careful as
to which nested sequences of subspaces are appropriate.  If $F$ is a subspace
of $V$ then $F^\perp$ denotes its annihilator in $W$. Similarly if ${'F}$
is a subspace of $W$ the ${'F}^\perp$ denotes its annihilator in $V$.  
We say that $F$ (resp. ${'F}$) is {\em closed} if $F = F^{\perp\perp}$
(resp. ${'F} = {'F}^{\perp\perp}$).  This is the closure relation in the
Mackey topology \cite{M}, i.e. the weak topology for the functionals on $V$ from
$W$ and on $W$ from $V$.
\smallskip

In order to avoid repeating the following definitions later on, we make them 
in somewhat greater generality than we need just now.

\begin{definition} \label{genflag}
{\rm Let $V$ and $W$ be countable dimensional vector spaces over a real
division ring
$\D = \R, \C \text{ or } \H$, with a nondegenerate bilinear pairing
$\langle \cdot , \cdot \rangle : V \times W \rightarrow \D$.  
A {\em chain} or {\em $\D$--chain} in $V$ (resp. $W$) is a set of 
$\D$--subspaces totally ordered by inclusion.  An {\em generalized $\D$--flag}
in $V$ (resp. $W$) is an $\D$--chain such that each subspace has an 
immediate predecessor or an immediate successor in the inclusion ordering, 
and every nonzero vector of $V$ (or $W$) is caught between an immediate 
predecessor successor (IPS) pair.  An generalized $\D$--flag $\cF$ in $V$
(resp. ${'\cF}$ in $W$) is {\em semiclosed} if 
$F \in \cF$ with $F \ne F^{\perp\perp}$ implies $\{F,F^{\perp\perp}\}$
is an IPS pair (resp. $'F \in {'\cF}$ with $'F \ne 'F^{\perp\perp}$ 
implies $\{'F,'F^{\perp\perp}\}$ is an IPS pair).}
\hfill $\diamondsuit$
\end{definition}

\begin{definition}\label{taut}
{\rm Let $\D$, $V$ and $W$ be as above.
Generalized $\D$--flags $\cF$ in $V$ and ${'\cF}$ in $W$ form a 
{\em taut couple} when
(i) if $F \in \cF$ then $F^\perp$ is
invariant by the $\ggl$--stabilizer of $'\cF$ and
(ii) if $'F \in {'\cF}$ then its annihilator
$'F^\perp$ is invariant by the $\ggl$--stabilizer of $\cF$.} 
\hfill $\diamondsuit$
\end{definition}

In the
$\gso$ and $\gsp$ cases one can use the associated bilinear form to
identify $V$ with $W$ and $\cF$ with ${'\cF}$.
Then we speak of a generalized flag $\cF$ in $V$ as {\em self--taut}.
If $\cF$ is a self--taut generalized flag in $V$ then \cite{DCPW}
every $F \in \cF$ is either isotropic or co--isotropic.
\smallskip

{\bf \begin{example}\label{lim-ord}
{\rm Here is a quick peek at an obvious phenomenon introduced by infinite
dimensionality.  Enumerate bases of $V = \C^\infty$ and 
$W = \C^\infty$ by $(\Z^+)^2$, say $\{v_i = v_{i_1,i_2}\}$ and
$\{w_j = w_{j_1,j_2}\}$, with $\langle v_i,w_j\rangle = 1$ if both
$i_1 = j_1$ and $i_2 = j_2$ and $\langle v_i,w_j\rangle = 0$ otherwise.  
Define $\cF = \{F_i\}$ 
ordered by inclusion where one builds up bases of the $F_i$ first with 
the $v_{i_1,1},\ i_1 \geqq 1$ and then the $v_{i_1,2},\ i_1 \geqq 1$ and 
then the $v_{i_1,3},\ i_1 \geqq 1$, and so on. One does the same for
$'\cF$ using the $\{w_j\}$.  Now these form a taut couple of semiclosed 
generalized flags
whose ordering involves an infinite number of limit ordinals.  
That makes it hard to use matrix methods.} \hfill $\diamondsuit$
\end{example}
}

\begin{theorem} \label{self-norm-cpx-parab}
The self--normalizing parabolic subalgebras of the Lie algebras $\gsl(V,W)$ and
$\ggl(V,W)$ are the normalizers of taut couples of semiclosed generalized flags
in $V$ and $W$, and this is a one to one correspondence.  
The self--normalizing parabolic subalgebras of $\gsp(V)$
are the normalizers of self--taut semiclosed generalized flags in $V$, and 
this too is a one to one correspondence.
\end{theorem}

\begin{theorem} \label{self-norm-cpx-parab-so}
The self--normalizing parabolic subalgebras of $\gso(V)$ are the 
normalizers of self--taut semiclosed generalized flags $\cF$ in $V$, and
there are two possibilities:
\begin{enumerate}
\item the flag $\cF$ is uniquely determined by the parabolic, or
\item there are exactly three self--taut generalized flags with the same 
stabilizer as $\cF$.
\end{enumerate}
The latter case occurs precisely when there exists an isotropic subspace $L \in \cF$ with  $\dim_\C L^\perp / L = 2$.  The three flags with the same stabilizer are then
\begin{itemize}
\item[] $\{F \in \cF \mid F \subset L \textrm{ or } L^\perp \subset F \}$
\item[] $\{F \in \cF \mid F \subset L \textrm{ or } L^\perp \subset F \} \cup M_1$
\item[] $\{F \in \cF \mid F \subset L \textrm{ or } L^\perp \subset F \} \cup M_2$
\end{itemize}
where $M_1$ and $M_2$ are the two maximal isotropic subspaces containing $L$.
\end{theorem}

{\bf \begin{example}\label{sl-in-gl}
{\rm Before proceeding we indicate an example which shows that not all parabolics
are equal to their normalizers.  Enumerate bases of $V = \C^\infty$ and 
$W = \C^\infty$ by rational numbers with pairing 

\centerline{$\langle v_q , w_r \rangle = 1 \text{ if } q > r, \, \, \, \,
= 0 \text{ if } q \leqq r$}

\noindent Then Span$\{v_q \otimes w_r \mid q \leqq r \}$ is a Borel 
subalgebra of $\ggl(\infty)$ contained in $\gsl(\infty)$.  This shows that 
$\gsl(\infty)$ is parabolic in $\ggl(\infty)$.} \hfill $\diamondsuit$
\end{example}
}

One pinpoints this situation as follows.  If $\gp$ is a (real or
complex) subalgebra of $\gg_\C$
and $\gq$ is a quotient algebra isomorphic to $\ggl(\infty;\C)$, say
with quotient map $f : \gp \to \gq$, then we refer to the composition
$trace\circ f : \gp \to \C$ as an {\em infinite trace} on $\gg_\C$.  If
$\{f_i\}$ is a finite set of infinite traces on $\gg_\C$ and $\{c_i\}$
are complex numbers, then we refer to the condition $\sum c_if_i = 0$ as
an {\em infinite trace condition} on $\gp$.  
\smallskip

These quotients can exist.  In Example \ref{lim-ord}
we can take $V_a$ to be the span of the $v_{i_1,a}$
and $W_a$ the span of the the dual $w_{i_1,a}$ for $a = 1, 2, ...$ and then
the normalizer of the taut couple $(\cF,{'\cF})$ has infinitely many
quotients $\ggl(V_a,W_a)$.

\begin{theorem} \label{gen-cpx-parab}
The parabolic subalgebras  $\gp$ in $\gg_\C$ are the algebras
obtained from self normalizing parabolics $\widetilde{\gp}$ by imposing 
infinite trace conditions.
\end{theorem}

As a general principle one tries to be explicit by constructing representations
that are as close to irreducible as feasible.  For this reason we will be
constructing principal series representations by inducing from parabolic 
subgroups that are minimal among the
self--normalizing parabolic subgroups.  Still, one should be aware of the
phenomenon of Example \ref{sl-in-gl} and Theorem \ref{gen-cpx-parab}.

\section{Real Parabolic Subalgebras and Subgroups}\label{sec3}
\setcounter{equation}{0}
In this section we discuss the structure of parabolic subalgebras of
real forms of the classical $\gsl(\infty,\C)$, $\gso(\infty,\C)$,
$\gsp(\infty,\C)$ and $\ggl(\infty,\C)$.  In this section $\gg_\C$ will
always be one of them and $G_\C$ will be the corresponding connected complex
Lie group.  Also, $\gg_\R$ will be a real form of $\gg_\C$, and
$G_\R$ will be the corresponding connected real subgroup of $G_\C$.
\smallskip

\begin{definition}\label{defrealp}
Let $\gg_\R$ be a real form of $\gg_\C$.  Then a subalgebra 
$\gp_\R \subset \gg_\R$ is a {\em parabolic subalgebra} if
its complexification $\gp_\C$ is a parabolic subalgebra of $\gg_\C$.
\hfill $\diamondsuit$
\end{definition}

When $\gg_\R$ has two inequivalent defining representations, in
other words when 
$$
\gg_\R = \gsl(\infty;\R),\,\, \ggl(\infty;\R),\,\, 
  \gsu(*,\infty),\,\, \gu(*,\infty),\,\, 
  \text{ or } \,\,\gsl(\infty;\H)
$$
we denote them by $V_\R$ and $W_\R$, and when $\gg_\R$ has only one 
defining representation, in other words when
$$
\gg_\R = \gso(*,\infty),\,\, \gsp(*,\infty),\,\,
   \gsp(\infty;\R),\,\, \text{ or } \,\,\gso^*(2\infty) 
   \text{ as quaternion matrices,}
$$
we denote it by $V_\R$.  The commuting algebra of $\gg_\R$ on $V_\R$ is a
real division algebra $\D$.  The main result of \cite{DCPW} is

\begin{theorem} \label{real-parab}
Suppose that $\gg_\R$ has two inequivalent defining representations.  Then
a subalgebra of $\gg_\R$ (resp. subgroup of $G_\R$) is
parabolic if and only if it is defined by infinite trace conditions
(resp. infinite determinant conditions) on the
$\gg_\R$--stabilizer (resp. $G_\R$--stabilizer) of
a taut couple of generalized $\D$--flags $\cF$ in $V_\R$ and $'\cF$ in $W_\R$.

Suppose that $\gg_\R$ has only one defining representation.
A subalgebra of $\gg_\R$ (resp. subgroup) of $G_\R$ is
parabolic if and only if it is defined by infinite trace conditions
(resp. infinite determinant conditions) on the
$\gg_\R$--stabilizer (resp. $G_\R$--stabilizer) of
a self--taut generalized $\D$--flag $\cF$ in $V_\R$.
\end{theorem}

\section{Levi Components of Complex Parabolics}\label{sec4}
\setcounter{equation}{0}
In this section we discuss Levi components of complex parabolic subalgebras,
recalling results from \cite{DP1}, \cite{DP2}, \cite{DC2}, \cite{DP3},
\cite{DC3} and \cite{W6}.  We start with the definition.

\begin{definition}\label{levi}
{\em Let $\gp$ be a locally finite Lie algebra
and $\gr$ its locally solvable radical.  A subalgebra $\gl \subset \gp$ is a
{\em Levi component} if $[\gp,\gp]$ is the semidirect sum
$(\gr \cap [\gp,\gp]) \subsetplus \gl$.}\hfill $\diamondsuit$
\end{definition}

Every finitary Lie algebra has a Levi
component.  Evidently, Levi components are maximal semisimple subalgebras,
but the converse fails for finitary Lie algebras.  In any case,
parabolic subalgebras of our classical Lie algebras $\gg_\C$ have maximal 
semisimple subalgebras, and those are their Levi components.

\begin{definition}\label{standard}
{\em Let $X \subset V$ and $Y \subset W$ be paired subspaces, isotropic in the
orthogonal and symplectic cases.  The subalgebras 
$$
\begin{aligned}
&\ggl(X,Y) \subset \ggl(V,W) \phantom{an}\text{ and }  \gsl(X,Y) \subset \gsl(V,W),\\
&\Lambda \ggl(X,Y) \subset \Lambda \ggl(V,V) \text{ and }
S\ggl(X,Y) \subset S\ggl(V,V)
\end{aligned}
$$ 
are called {\em standard}.}\hfill $\diamondsuit$
\end{definition}

\begin{proposition}\label{struc-levi}
A subalgebra $\gl_\C \subset \gg_\C$ is the Levi
component of a parabolic subalgebra of $\gg_\C$
if and only if it is the direct sum of standard special linear
subalgebras and at most one subalgebra $\Lambda \ggl(X,Y)$ in the orthogonal case,
at most one subalgebra $S\ggl(X,Y)$ in the symplectic case.
\end{proposition}

The occurrence of ``at most one subalgebra'' in Proposition \ref{struc-levi}
is analogous to the finite dimensional case, where it is seen by deleting some
simple root nodes from a Dynkin diagram.
\smallskip

Let $\gp$ be the parabolic subalgebra of $\gsl(V, W)$ or $\ggl(V, W)$ defined
by the taut couple $(\cF, {'\cF})$ of semiclosed generalized flags.  Denote
\addtocounter{equation}{1}
\begin{equation}\label{J}\tag{\thesection.\arabic{equation}}
\begin{aligned}
&J = \{(F',F'') \text{ IPS pair in } \cF \mid F' = (F')^{\perp\perp}
	\text{ and } \dim F''/F' > 1\},\\
&'J = \{('F',{'F''}) \text{ IPS pair in } {'\cF} \mid {'F}' = ('F')^{\perp\perp},
         \dim {'F''}/{'F'} > 1\}.
\end{aligned}
\end{equation}
Since $V \times W \to \C$ is nondegenerate the sets $J$ and $'J$ are in
one to one correspondence by: $(F''/F') \times ({'F''}/{'F'}) \to \C$ is
nondegenerate.  We use this to identify $J$ with $J'$, and we write
$(F_j',F_j'')$ and $('F_j',{'F_j''})$ treating $J$ as an index set.

\begin{theorem}\label{glpar}
Let $\gp$ be the parabolic subalgebra of $\gsl(V, W)$ or $\ggl(V, W)$ defined
by the taut couple $\cF$ and $'\cF$ of semiclosed generalized flags.  For each
$j \in J$ choose a subspace $X_j \subset V$ and a subspace $Y_j \subset W$
such that $F_j'' = X_j + F_j'$ and $'F_j'' = Y_j + {'F_j}'$
Then $\bigoplus_{j \in J}\, \gsl(X_j,Y_j)$ is a Levi component of $\gp$. The
inclusion relations of $\cF$ and $'\cF$ induce a total order on $J$.

Conversely, if $\gl$ is a Levi component of $\gp$ then there exist subspaces
$X_j \subset V$ and $Y_j \subset W$ such that 
$\gl = \bigoplus_{j \in J}\, \gsl(X_j,Y_j)$.
\end{theorem}

Now the idea of finite matrices with blocks down the diagonal suggests the
construction of $\gp$ from the totally ordered set $J$ and the direct sum
$\gl = \bigoplus_{j \in J}\, \gsl(X_j,Y_j)$ of standard special linear 
algebras.  We outline the idea of the construction; see \cite{DC3}.  First, 
$\langle X_j, Y_{j'}\rangle = 0$ for $j \ne j'$ because the
$\gs_j = \gsl(X_j,Y_j)$ commute with each other.  Define
$U_j := (( \bigoplus_{k \leqq j}\, X_k)^\perp \oplus Y_j)^\perp$. 
Then one proves $U_j = ((U_j \oplus X_j)^\perp \oplus Y_j)^\perp$.
From that, one shows that there is a unique semiclosed generalized flag 
$\cF_{min}$ in $V$ with the same stabilizer as the set 
$\{U_j, U_j \oplus X_j\,  |\,  j \in J\}$. 
One constructs similar subspaces $'U_j \subset W$ and shows that there is a 
unique semiclosed generalized flag $'\cF_{min}$ in $W$ with the same stabilizer 
as the set $\{'U_j, {'U}_j \oplus Y_j \, |\,  j \in J\}$.  In fact
$(\cF_{min} , {'\cF}_{min})$ is the minimal taut couple with IPS
pairs $U_j \subset (U_j \oplus X_j)$ in $\cF_0$ and 
$(U_j \oplus X_j)^\perp \subset ((U_j \oplus X_j)^\perp \oplus Y_j)$ in ${'\cF}_0$ 
for $j \in J$.  If $(\cF_{max}, {'\cF}_{max})$ is maximal among the taut couples of
semiclosed generalized flags with IPS pairs $U_j \subset (U_j \oplus X_j)$ 
in $\cF_{max}$
and $(U_j \oplus X_j)^\perp \subset ((U_j \oplus X_j)^\perp \oplus Y_j)$ in 
${'\cF}_{max}$ then the corresponding parabolic $\gp$ has Levi component $\gl$.
\smallskip

The situation is essentially the same for Levi components of
parabolic subalgebras of $\gg_\C = \gso(\infty;\C) \text{ or }
\gsp(\infty;\C)$, except that we modify the definition (\ref{J}) of $J$ to add
the condition that $F''$ be isotropic, and we add the orientation aspect
of the $\gso$ case.

\begin{theorem}\label{sosplevi}
Let $\gp$ be the parabolic subalgebra of $\gg_\C = \gso(V)$ or $\gsp(V)$,
defined by the self--taut semiclosed generalized flag $\cF$.  Let 
$\widetilde{F}$ be the union of all subspaces $F''$ in IPS pairs 
$(F',F'')$ of $\cF$ for which $F''$ is isotropic.  Let $\widetilde{'F}$ be the 
intersection of all subspaces $F'$ in IPS pairs for which $F'$ is
closed ($F' = (F')^{\perp\perp}$) and coisotropic.
Then $\gl$ is a Levi component of $\gp$ if and
only if there are isotropic subspaces $X_j, Y_j$ in $V$ such that
$$
\text{$F''_j =F'_j + X_j$ and ${'F''_j} ={'F_j} +Y_j$ for every $j \in J$}
$$
and a subspace $Z$ in $V$ such that
$\widetilde{F} = Z + \widetilde{'F}$,
where $Z = 0$ in case $\gg_\C = \gso(V)$ and 
$\dim \widetilde{F}/\widetilde{'F} \leqq 2$, such that 
$$
\begin{aligned}
&\gl = \gsp(Z) \oplus {\bigoplus}_{j \in J}\ \gsl(X_j,Y_j) \text{ if }
	\gg_\C = \gsp(V),\\
&\gl = \gso(Z) \oplus {\bigoplus}_{j \in J}\ \gsl(X_j,Y_j) \text{ if }
	\gg_\C = \gso(V).
\end{aligned}
$$
Further, the inclusion relations of $\cF$ induce a 
total order on $J$ which leads to a construction of $\gp$ from $\gl$.
\end{theorem}

\section{Chevalley Decomposition}\label{sec5}
\setcounter{equation}{0}
In this section we apply the extension \cite{DC2} to our parabolic subalgebras,
of the Chevalley decomposition for a (finite dimensional) algebraic Lie algebra.
\smallskip

Let $\gp$ be a locally finite linear Lie algebra, in our case a subalgebra 
of $\ggl(\infty)$.  Every element $\xi \in \gp$ has a Jordan canonical form,
yielding a decomposition $\xi = \xi_{ss} + \xi_{nil}$ into semisimple and
nilpotent parts.  The algebra $\gp$ is {\em splittable} if it contains the 
semisimple and the nilpotent parts of each of its elements.  Note that
$\xi_{ss}$ and $\xi_{nil}$ are polynomials in $\xi$; this follows from the
finite dimensional fact.  In particular, if $X$ is any $\xi$--invariant
subspace of $V$ then it is invariant under both $\xi_{ss}$ and $\xi_{nil}$.
\smallskip

Conversely, parabolic subalgebras (and many others) of our classical 
Lie algebras $\gg$ are splittable.
\smallskip

The {\em linear nilradical} of a subalgebra $\gp \subset \gg$ is the set 
$\gp_{nil}$ of all nilpotent elements of the locally solvable radical $\gr$
of $\gp$.  It is a locally nilpotent ideal in $\gp$ and satisfies
$\gp_{nil} \cap [\gp, \gp] = \gr \cap [\gp, \gp]$.
\smallskip

If $\gp$ is splittable then it has a well defined maximal locally reductive 
subalgebra $\gp_{red}$.  This means that $\gp_{red}$ is an increasing union of
finite dimensional reductive Lie algebras, each reductive in the next.
In particular $\gp_{red}$ maps isomorphically under the projection 
$\gp \to \gp/\gp_{nil}$.  That gives a semidirect sum decomposition
$\gp = \gp_{nil} \subsetplus \gp_{red}$ analogous to the Chevalley
decomposition mentioned above.  Also, here, 
$$
\gp_{red} = \gl\subsetplus\gt \quad \text{ and } \quad [\gp_{red},\gp_{red}] = \gl
$$
where $\gt$ is a toral subalgebra and $\gl$ is the Levi component of $\gp$.
A glance at $\gu(\infty)$ or $\ggl(\infty;\C)$ shows that the semidirect 
sum decomposition of $\gp_{red}$ need not be direct.

\section{Levi and Chevalley Components of Real Parabolics}\label{sec6}
\setcounter{equation}{0}
Now we adapt the material of Sections \ref{sec4} and \ref{sec5} to study Levi 
and Chevalley components of
real parabolic subalgebras in the real classical Lie algebras.  
\smallskip

Let $\gg_\R$ be a real form of a classical locally finite complex simple 
Lie algebra $\gg_\C$.  Consider a real parabolic subalgebra $\gp_\R$.  It
has form $\gp_\R = \gp_\C \cap \gg_\R$ where its complexification $\gp_\C$
is parabolic in $\gg_\C$.  Let $\tau$ denote complex conjugation of
$\gg_\C$ over $\gg_\R$.  Then the locally solvable radical $\gr_\C$ of
$\gp_\C$ is $\tau$--stable because $\gr_\C + \tau\gr_\C$ is a locally
solvable ideal, so the locally solvable radical $\gr_\R$ of $\gp_\R$ is
a real form of $\gr_\C$.
\smallskip

Let $\gl_\R$ be a maximal semisimple subalgebra of $\gp_\R$.  Its
complexification $\gl_\C$ is a maximal semisimple subalgebra, hence a Levi
component, of $\gp_\C$.  Thus $[\gp_\C,\gp_\C]$ is the semidirect
sum $(\gr_\C \cap [\gp_\C,\gp_\C]) \subsetplus \gl_\C$.  The elements of
this formula all are $\tau$--stable, so we have proved

\begin{lemma}\label{real-levi-1}
The Levi components of $\gp_\R$ are real forms of the Levi components
of $\gp_\C$.
\end{lemma}

If $\gg_\C$ is $\gsl(V,W)$ or $\ggl(V,W)$ as in 
Theorem \ref{glpar}, then $\gl_\C = \bigoplus_{j \in J}\, \gsl(X_j,Y_j)$
as indicated there.  Initially the possibilities for the action of $\tau$ are
\begin{itemize}
\item $\tau$ preserves $\gsl(X_j,Y_j)$ with fixed point set 
	$\gsl(X_{j,\R},Y_{j,\R}) \cong \gsl(*;\R)$,
\item $\tau$ preserves $\gsl(X_j,Y_j)$ with fixed point set
        $\gsl(X_{j,\H},Y_{j,\H}) \cong \gsl(*;\H)$, 
\item $\tau$ preserves $\gsl(X_j,Y_j)$ with f.p. set
        $\gsu(X'_j,X_j'') \cong \gsu(*,*)$, $X_j = X'_j + X_j''$, and
\item $\tau$ interchanges two summands $\gsl(X_j,Y_j)$ and
        $\gsl(X_{j'},Y_{j'})$ of $\gl_\C$, with fixed point set 
        the diagonal ($\cong \gsl(X_j,Y_j)$) of their direct sum.
\end{itemize}
If $\gg_\C = \gso(V)$ as in Theorem \ref{sosplevi}, $\gl_\C$ can also have a
summand $\gso(Z)$, or if $\gg_\C = \gsp(V)$ it can also have a summand
$\gsp(V)$.  Except when $A_4 = D_3$ occurs, these additional summands must
be $\tau$--stable, resulting in fixed point sets
\begin{itemize}
\item when $\gg_\C = \gso(V)$: $\gso(Z)^\tau$ is $\gso(*,*)$ or $\gso^*(2\infty)$,
\item when $\gg_\C = \gsp(V)$: $\gsp(Z)^\tau$ is $\gsp(*,*)$ or $\gsp(*;\R)$.
\end{itemize}

\section{Minimal Parabolic Subgroups}\label{sec7}
\setcounter{equation}{0}
We describe the structure of minimal parabolic subgroups of the
classical real simple Lie groups $G_\R$.  

\begin{proposition}\label{minlevi}
Let $\gp_\R$ be a parabolic subalgebra of $\gg_\R$ and $\gl_\R$ a Levi
component of $\gp_\R$.  If $\gp_\R$ is a minimal parabolic subalgebra
then $\gl_\R$ is a direct sum of
finite dimensional compact algebras $\gsu(p)$, $\gso(p)$ and $\gsp(p)$,
and their infinite dimensional 
limits $\gsu(\infty)$, $\gso(\infty)$ and $\gsp(\infty)$.
If $\gl_\R$ is a direct sum of
finite dimensional compact algebras $\gsu(p)$, $\gso(p)$ and $\gsp(p)$ and
their limits $\gsu(\infty)$, $\gso(\infty)$ and $\gsp(\infty)$, 
then $\gp_\R$ contains a minimal parabolic subalgebra of $\gg_\R$ with
the same Levi component $\gl_\R$.
\end{proposition}

\begin{proof} Suppose that $\gp_\R$ is a minimal parabolic subalgebra
of $\gg_\R$.  If a direct summand $\gl'_\R$ of  $\gl_\R$ has a
proper parabolic subalgebra $\gq_\R$, we replace $\gl'_\R$ by $\gq_\R$
in $\gl_\R$ and $\gp_\R$.  In other words we refine the flag(s) that define 
$\gp_\R$.  The refined flag defines a parabolic $\gq_\R \subsetneqq \gp_\R$.
This contradicts minimality.  Thus no summand of $\gl_\R$ has a proper 
parabolic subalgebra.
Theorems \ref{glpar} and \ref{sosplevi} show that $\gsu(p)$, $\gso(p)$ and 
$\gsp(p)$, and their limits $\gsu(\infty)$, $\gso(\infty)$ and $\gsp(\infty)$,
are the only possibilities for the simple summands of $\gl_\R$.
\smallskip

Conversely suppose that the summands of $\gl_\R$ are $\gsu(p)$, $\gso(p)$ and 
$\gsp(p)$ or their limits $\gsu(\infty)$, $\gso(\infty)$ and $\gsp(\infty)$.
Let $(\cF, {'\cF})$ or $\cF$ be the flag(s) that define $\gp_\R$.
In the discussion between Theorems \ref{glpar} and \ref{sosplevi} we described
a a minimal taut couple $(\cF_{min}, {'\cF}_{min})$ and a maximal taut couple 
$(\cF_{max}, {'\cF}_{max})$ (in the $\gsl$ and $\ggl$ cases) 
of semiclosed generalized flags which define parabolics that have the same
Levi component $\gl_\C$ as $\gp_\C$.  By construction $(\cF, {'\cF})$ refines
$(\cF_{min}, {'\cF}_{min})$ and $(\cF_{max}, {'\cF}_{max})$ refines 
$(\cF, {'\cF})$.  As $(\cF_{min}, {'\cF}_{min})$ is uniquely defined from
$(\cF, {'\cF})$ it is $\tau$--stable.  Now the maximal $\tau$--stable
taut couple $(\cF^*_{max}, {'\cF}^*_{max})$ of semiclosed generalized flags
defines a $\tau$--stable parabolic $\gq_\C$ with the same Levi component $\gl_\C$ 
as $\gp_\C$, and $\gq_\R := \gq_\C\cap\gg_\R$ is a minimal parabolic
subalgebra of $\gg_\R$ with Levi component $\gl_\R$.
\smallskip

The argument is the same when $\gg_\C$ is $\gso$ or $\gsp$. 
\end{proof}

Proposition \ref{minlevi} says that the Levi components of the minimal 
parabolics are the compact real forms, in the sense of \cite{S}, of the complex
$\gsl$, $\gso$ and $\gsp$.  We extend this notion.  
\smallskip

The group $G_\R$ has the natural {\em Cartan involution} $\theta$ such that
$d\theta((\gp_\R)_{red}) = (\gp_\R)_{red}$, defined as follows.  Every element
of $\gl_\R$ is elliptic, and $(\gp_\R)_{red} = \gl_\R \subsetplus \gt_\R$ where
$\gt_\R$ is toral, so every element of $(\gp_\R)_{red}$ is semisimple.
(This is where we use minimality of the parabolic $\gp_\R$.)
Thus $(\gp_\R)_{red}\cap \gg_{n,\R}$ is reductive in $\gg_{m,\R}$ 
for every $m \geqq n$.  Consequently we have Cartan involutions $\theta_n$ of
the groups $G_{n,\R}$ such that $\theta_{n+1}|_{G_{n,\R}} = \theta_n$
and  $d\theta_n((\gp_\R)_{red}\cap \gg_{n,\R}) = 
(\gp_\R)_{red}\cap \gg_{n,\R}$.  Now $\theta = \varinjlim \theta_n$
(in other words $\theta|_{G_{n,\R}} = \theta_n$)
is the desired Cartan involution of $\gg_\R$.  Note that $\gl_\R$ is 
contained in the fixed point set of $d\theta$.
\smallskip

The Lie algebra $\gg_\R = \gk_\R + \gs_\R$ where $\gk_\R$ is the
$(+1)$--eigenspace of $d\theta$ and $\gs_\R$ is the $(-1)$--eigenspace.
The fixed point set $K_\R = G_\R^\theta$ is the direct limit of the maximal 
compact subgroups $K_{n,\R} = G_{n,\R}^{\theta_n}$.  We will refer to $K_\R$ as
a {\em maximal lim--compact subgroup} of $G_\R$ and to $\gk_\R$ as a maximal
{\em lim--compact subalgebra} of $\gg_\R$.  By construction $\gl_\R \subset
\gk_\R$, as in the case of finite dimensional minimal parabolics.
Also as  in the finite dimensional
case (and using the same proof), $[\gk_\R,\gk_\R] \subset \gk_\R$,
$[\gk_\R,\gs_\R] \subset \gs_\R$ and $[\gs_\R,\gs_\R] \subset \gk_\R$.

\begin{lemma}\label{construct-ma}
Decompose $(\gp_\R)_{red} = \gm_\R + \ga_\R$ where 
$\gm_\R = (\gp_\R)_{red}\cap \gk_\R$ and $\ga_\R = (\gp_\R)_{red}\cap \gs_\R$.
Then $\gm_\R$ and $\ga_\R$ are ideals in $(\gp_\R)_{red}$ with $\ga_\R$ 
commutative.
In particular $(\gp_\R)_{red} = \gm_\R \oplus \ga_\R$, direct sum of ideals.
\end{lemma}

\begin{proof}
Since $\gl_\R = [(\gp_\R)_{red}, (\gp_\R)_{red}]$  we compute 
$[\gm_\R,\ga_\R] \subset \gl_\R \cap \ga_\R = 0$.
In particular $[[\ga_\R, \ga_\R],\ga_\R] = 0$.  So $[\ga_\R, \ga_\R]$
is a commutative ideal in the semisimple algebra $\gl_\R$, in other words
$\ga_\R$ is commutative.
\end{proof}

The main result of this section is the following generalization of the
standard decomposition of a finite dimensional real parabolic.  We have
formulated it to emphasize the parallel with the finite dimensional case.
However some details of the construction are rather different; see 
Proposition \ref{construct-p} and the discussion leading up to it.

\begin{theorem}\label{lang-alg}
The minimal parabolic subalgebra $\gp_\R$ of $\gg_\R$ decomposes as
$\gp_\R = \gm_\R + \ga_\R + \gn_\R =
\gn_\R \subsetplus (\gm_\R \oplus \ga_\R)$, where $\ga_\R$ is commutative,
the Levi component $\gl_\R$ is an ideal in $\gm_\R$\,, and $\gn_\R$ is 
the linear nilradical $(\gp_\R)_{nil}$.  On the group level,
$P_\R = M_\R A_\R N_\R = N_\R \ltimes (M_\R \times A_\R)$ where
$N_\R = \exp(\gn_\R)$ is the linear unipotent radical of $P_\R$, 
$A_\R = \exp(\ga_\R)$ is isomorphic to a vector group, and 
$M_\R = P_\R \cap K_\R$ is limit--compact with Lie algebra $\gm_\R$\, .
\end{theorem}

\begin{proof}
The algebra level statements come out of Lemma \ref{construct-ma} and the
semidirect sum decomposition $\gp_\R = (\gp_\R)_{nil} \subsetplus (\gp_\R)_{red}$.
\smallskip

For the group level statements, we need only check that $K_\R$ meets every
topological component of $P_\R$.  Even though $P_\R\cap G_{n,\R}$ need not
be parabolic in $G_{n,\R}$, the group $P_\R\cap\theta P_\R\cap G_{n,\R}$ is
reductive in $G_{n,\R}$ and $\theta_n$--stable, so $K_{n,\R}$ meets each of
its components.  Now $K_\R$ meets every component of $P_\R\cap\theta P_\R$.
The linear unipotent radical of $P_\R$ has Lie algebra $\gn_\R$ and thus
must be equal to $\exp(\gn_\R)$, so it does not effect components.  Thus
every component of $P_{red}$ is represented by an element of 
$K_\R \cap P_\R\cap\theta P_\R = K_\R \cap P_\R = M_\R$.  That derives
$P_\R = M_\R A_\R N_\R = N_\R \ltimes (M_\R \times A_\R)$ from
$\gp_\R = \gm_\R + \ga_\R + \gn_\R = \gn_\R \subsetplus (\gm_\R \oplus \ga_\R)$.
\end{proof}

The reductive part of the group $\gp_\R$ can be constructed explicitly.
We do this for the cases where $\gg_\R$ is defined by a hermitian 
form $f: V_\F \times V_\F \to \F$ where $\F$ is $\R$, $\C$ 
or $\H$.  The idea is the same for the other cases.
See Proposition \ref{construct-p} below.
\smallskip

Write $V_\F$ for $V_\R$, $V_\C$ or $V_\H$, as appropriate, 
and similarly for $W_\F$.  We use $f$ for an 
$\F$--conjugate--linear identification of $V_\F$ and $W_\F$.
We are dealing with a minimal Levi component $\gl_\R =
\bigoplus_{j \in J}\, \gl_{j,\R}$ where the $\gl_{j,\R}$ are simple.  
Let $X_\F$ denote the sum of the
corresponding subspaces $(X_j)_\F \subset V_\F$ and $Y_\F$ the
analogous sum of the $(Y_j)_\F \subset W_\F$.  Then $X_\F$ and $Y_\F$
are nondegenerately paired.  Of course they may be small, even zero.
In any case, 
\addtocounter{equation}{1}
\begin{equation}\label{fill-out-VW}\tag{\thesection.\arabic{equation}}
\begin{aligned}
&V_\F = X_\F \oplus Y_\F^\perp \, ,
W_\F = Y_\F \oplus X_\F^\perp, \text{ and } \\
&X_\F^\perp \text{ and } Y_\F^\perp
\text{ are nondegenerately paired.}
\end{aligned}
\end{equation}
These direct sum decompositions (\ref{fill-out-VW}) now become
\addtocounter{equation}{1}
\begin{equation}\label{fill-out-V}\tag{\thesection.\arabic{equation}}
V_\F = X_\F \oplus X_\F^\perp \quad \text{ and } \quad f 
\text{ is nondegenerate on each summand.} 
\end{equation}
Let $X'$ and $X''$ be paired maximal isotropic subspaces of $X_\F^\perp$.  Then 
\addtocounter{equation}{1}
\begin{equation}\label{expand-V}\tag{\thesection.\arabic{equation}}
V_\F = X_\F \oplus (X'_\F \oplus X''_\F) \oplus Q_\F \text{ where }
Q_\F := (X_\F \oplus (X'_\F \oplus X''_\F))^\perp .
\end{equation}

The subalgebra 
$\{\xi \in \gg_\R \mid \xi(X_\F \oplus Q_\F) = 0\}$
of $\gg_\R$ has a maximal toral subalgebra $\ga\dagger_\R$, contained in 
$\gs_\R$, 
in which every element has all eigenvalues real.  One example, which is
diagonalizable (in fact diagonal) over $\R$, is
\addtocounter{equation}{1}
\begin{equation}\label{def-a'}\tag{\thesection.\arabic{equation}}
\begin{aligned}
\ga^\dagger_\R = {\bigoplus}_{\ell \in C}\, \ggl(x_\ell'\R,x_\ell''\R) 
	&\text{ where }\\
	& \{x'_\ell \mid \ell \in C\} \text{ is a basis of } X'_\F 
	\text{ and } \\
	& \{x''_\ell \mid \ell \in C\} \text{ is the dual basis of } X''_\F.
\end{aligned}
\end{equation}
We interpolate the self--taut semiclosed generalized flag $\cF$ defining $\gp$
with the subspaces $x_\ell'\R \oplus x_\ell''\R$.  Any such interpolation
(and usually there will be infinitely many) gives a self--taut semiclosed
generalized flag $\cF^\dagger$ and defines a minimal self--normalizing 
parabolic subalgebra $\gp^\dagger_\R$ of $\gg_\R$ with the same Levi component 
as $\gp_\R$.  The decompositions corresponding to (\ref{fill-out-VW}),
(\ref{fill-out-V}) and (\ref{expand-V}) are given by 
$X_\F^\dagger = X_\F \oplus (X'_\F \oplus X''_\F)$ and $Q_\F^\dagger = Q_\F$.
\smallskip

In addition, the subalgebra
$\{\xi \in \gp_\R \mid \xi(X_\F  \oplus (X'_\F \oplus X''_\F)) = 0\}$
has a maximal toral subalgebra $\gt'_\R$ in which every eigenvalue is pure 
imaginary, because $f$ is definite on $Q_\F$.  It is
unique because it has derived algebra zero and is given by the action of the 
$\gp_\R$--stabilizer of $Q_\F$ on the definite subspace $Q_\F$.  
This uniqueness tell us that $\gt'_\R$ is the same for $\gp_\R$ and 
$\gp^\dagger_\R$.
\smallskip

Let $\gt''_\R$ denote the maximal toral subalgebra in
$\{\xi \in \gp_\R \mid \xi(X_\F \oplus Q_\F)) = 0\}$.  It stabilizes each 
Span($x'_\ell ,x''_\ell$) in (\ref{def-a'}) and centralizes $\ga^\dagger_\R$,
so it vanishes if $\F \ne \C$.  The $\gp_\R^\dagger$ analog of $\gt''_\R$
is $0$ because $X^\dagger_\F \oplus Q_\F = 0$.  In any case we have
\addtocounter{equation}{1}
\begin{equation}\label{def-t}\tag{\thesection.\arabic{equation}}
\gt_\R = \gt_\R^\dagger := \gt'_\R \oplus \gt''_\R\,.
\end{equation}

For each $j \in J$ we define
an algebra that contains $\gl_{j,\R}$ and acts on $(X_j)_\F$ by:
if $\gl_{j,\R} = \gsu(*)$ then $\widetilde{\gl_{j,\R}} = \gu(*)$ (acting
        on $(X_j)_\C$); otherwise $\widetilde{\gl_{j,\R}} = \gl_{j,\R}$.
Define 
\addtocounter{equation}{1}
\begin{equation}\label{def-m'}\tag{\thesection.\arabic{equation}}
\widetilde{\gl_\R} = 
  \bigoplus_{j \in J}\, \widetilde{\gl_{j,\R}} \quad
\text{ and } \quad \gm^\dagger_\R = \widetilde{\gl_\R} + \gt_\R\, .
\end{equation}
Then, by construction, $\gm^\dagger_\R = \gm_\R$.  Thus 
$\gp^\dagger_\R$ satisfies
\addtocounter{equation}{1}
\begin{equation}\label{def-p'}\tag{\thesection.\arabic{equation}}
\gp^\dagger_\R := \gm_\R +\ga^\dagger_\R + \gn^\dagger_\R=
	\gn^\dagger_\R \subsetplus (\gm_\R \oplus \ga^\dagger_\R).
\end{equation}
Let $\gz_\R$ denote the centralizer of $\gm_\R \oplus \ga_\R$ in $\gg_\R$
and let $\gz^\dagger_\R$ denote the centralizer of 
$\gm_\R \oplus \ga^\dagger_\R$ in $\gg_\R$.  We claim
\addtocounter{equation}{1}
\begin{equation}\label{red-parts}\tag{\thesection.\arabic{equation}}
\gm_\R + \ga_\R = \widetilde{\gl_\R} + \gz_\R \text{ and }
	\gm_\R + \ga^\dagger_\R = \widetilde{\gl_\R} + \gz^\dagger_\R
\end{equation}
For by construction $\gm_\R \oplus \ga_\R =
\widetilde{\gl_\R} + \gt_\R + \ga_\R \subset \widetilde{\gl_\R} + \gz_\R$.
Conversely if $\xi \in \gz_\R$ it preserves each $X_{j,\F}$, each joint
eigenspace of $\ga_\R$ on $X'_\F \oplus X''_\F$, and each joint
eigenspace of $\gt_\R$, so $\xi \subset \widetilde{\gl_\R} + \gt_\R + \ga_\R$.
Thus $\gm_\R + \ga_\R = \widetilde{\gl_\R} + \gz_\R$.  The same argument
shows that $\gm_\R + \ga^\dagger_\R = \widetilde{\gl_\R} + \gz^\dagger_\R$.
\smallskip

If $\ga_\R$ is diagonalizable as in the definition (\ref{def-a'}) of 
$\ga^\dagger_\R$, in other words if it is a sum of standard $\ggl(1;\R)$'s,
then we could choose $\ga_\R^\dagger = \ga_\R$, hence could
construct $\cF^\dagger$ equal to $\cF$, resulting in $\gp_\R = \gp_\R^\dagger$.
In summary:

\begin{proposition}\label{construct-p}
Let $\gg_\R$ be defined by a hermitian form and let $\gp_\R$ be a minimal
self--normalizing parabolic subalgebra. In the notation above, $\gp^\dagger_\R$
is a minimal self--normalizing parabolic subalgebra of $\gg_\R$ with 
$\gm^\dagger_\R = \gm_\R$.  In particular $\gp^\dagger_\R$ and $\gp_\R$ 
have the same Levi component.  Further we can take $\gp_\R = \gp^\dagger_\R$ 
if and only if $\ga_\R$ is the sum of commuting standard $\ggl(1;\R)$'s.
\end{proposition}

Similar arguments give the construction behind Proposition \ref{construct-p} 
for the other real simple direct limit Lie algebras.

\section{The Inducing Representation} \label{sec8}
\setcounter{equation}{0}
In this section $P_\R$ is a self normalizing minimal parabolic subgroup
of $G_\R$.  We discuss representations of $P_\R$ and the induced 
representations of $G_\R$.  The latter are the {\em principal series} 
representations of $G_\R$ associated to $\gp_\R$, or more precisely to
the pair $(\gl_\R,J)$ where $\gl_\R$ is the Levi component and $J$ is
the ordering on the simple summands of $\gl_\R$.
\smallskip

We must first choose a class $\cC_{M_\R}$ of representations of $M_\R$.
Reasonable choices include various classes of unitary representations 
(we will discuss this in a moment) and continuous 
representations on nuclear Fr\' echet spaces, but ``tame'' (essentially 
the same as $II_1$) may be the best with which to start.  In any case,
given a representation $\kappa$ in our chosen class and a linear functional
$\sigma: \ga_\R \to \R$ we have the representation $\kappa\otimes e^{i\sigma}$
of $M_\R \times A_\R$.  Here $e^{i\sigma}(a)$ means $e^{i\sigma(\log a)}$
where $\log : A_\R \to \ga_\R$ inverts $\exp : \ga_\R \to A_\R$.  We
write $E_\kappa$ for the representation space of $\kappa$.
\smallskip

We discuss some possibilities for $\cC_{M_\R}$.
 Note that $\gl_\R = [(\gp_\R)_{red},(\gp_\R)_{red}]
= [\gm_\R + \ga_\R, \gm_\R + \ga_\R] = [\gm_\R,\gm_\R]$.  
Define
$$
L_\R = [M_\R,M_\R] \text{ and } T_\R = M_\R/L_\R\,.
$$
Then $T_\R$ is a real toral group with all eigenvalues pure imaginary,
and $M_\R$ is an extension
$
1 \to L_\R \to M_\R \to T_\R \to 1 \,.
$
Examples indicate that $M_\R$ is the product of a closed subgroup $T'_\R$
of $T_\R$ with factors of the group $L'_\R$ indicated in the previous section.
That was where we replaced summands $\gsu(*)$ of $\gl_\R$ by slightly
larger algebras $\gu(*)$, hence subgroups $SU(*)$ of $L_\R$ by slightly
larger groups $U(*)$.  There is no need to discuss the representations of the
classical finite dimensional $U(n)$, $SO(n)$ or $Sp(n)$, where we have the
Cartan highest weight theory and other classical combinatorial methods.  So
we look at $U(\infty)$.
\smallskip

{\bf Tensor Representations of $U(\infty)$}.  In the classical setting,
one can use the action of the symmetric group $\gS_n$, permuting factors
of $\otimes^n(\C^p)$. This gives a representation of $U(p) \times \gS_n$.
Then we have the action of $U(p)$ on tensors picked out by an irreducible 
summand of that action of $\gS_n$.  These summands occur with multiplicity $1$.
See Weyl's book \cite{W}.  
Segal \cite{Se}, Kirillov \cite{K}, and Str\u atil\u a \& Voiculescu
\cite{SV}  developed and proved an analog of this for $U(\infty)$.  
However those ``tensor representations'' form a small class of the continuous 
unitary representations of $U(\infty)$.  They are factor representations
of type $II_\infty$, but they are somewhat restricted in that
they do not even extend to the class of unitary operators of the form 
$1 + \text{(compact)}$.  See \cite[Section 2]{SV2} for a summary of this topic.
Because of this limitation one may also wish to consider other classes of factor 
representations of $U(\infty)$.
\smallskip

{\bf Type $II_1$ Representations of $U(\infty)$}.
Let $\pi$ be a continuous unitary finite factor representation of $U(\infty)$.
It has a character $\chi_\pi(x) = \trace \pi(x)$ (normalized trace).
Voiculescu \cite{V} worked out the parameter space for these finite
factor representations.  It consists of all bilateral sequences
$\{c_n\}_{-\infty < n < \infty}$ such that 
(i) $\det((c_{m_i + j - i})_{1 \leqq i, j \leqq N} \geqq 0$ for
          $m_i \in \Z$ and $N \geqq 0$ and 
(ii) $\sum c_n = 1$.
The character corresponding to $\{c_n\}$ and $\pi$ is 
$\chi_\pi(x) = \prod_i p(z_i)$ where $\{z_i\}$ is the multiset of
eigenvalues of $x$ and 
$p(z) = \sum c_nz^n$. Here $\pi$ extends to the group of all unitary operators 
$X$ on the Hilbert space completion of $\C^\infty$ such that $X - 1$ is
of trace class.  See \cite[Section 3]{SV2} for a more detailed summary.
This may be the best choice of class $\cC_{M_\R}$.  It is closely tied to
the Olshanskii--Vershik notion (see \cite{O}) of tame representation.
\smallskip

{\bf Other Factor Representations of $U(\infty)$}.
Let $\cH$ be the Hilbert space completion of $\varinjlim \cH_n$ where
$\cH_n$ is the natural representation space of $U(n)$.  Fix a bounded
hermitian operator $B$ on $\cH$ with $0 \leqq B \leqq I$.  Then
$$
\psi_B : U(\infty) \to \C\,, \text{ defined by }
\psi_B(x) = \det((1 - B) + B x)
$$
is a continuous  function of positive type on $U(\infty)$.  Let $\pi_B$
denote the associated cyclic representation of $U(\infty)$.  Then
(\cite[Theorem 3.1]{SV3}, or see \cite[Theorem 7.2]{SV2}),
\begin{itemize}
\item[(1)] $\,\,\psi_B$ is of type $I$ if and only if $B(I-B)$ is of trace class.
	In that case $\pi_B$ is a direct sum of irreducible representations.
\item[(2)] $\,\,\psi_B$ is factorial and type $I$ if and only if $B$ is a 
	projection.  In that case $\pi_B$ is irreducible.
\item[(3)] $\,\,\psi_B$ is factorial but not of type $I$ if and only if
	$B(I-B)$ is not of trace class.  In that case
\begin{itemize}
  \item[(i)] $\,\,\psi_B$ is of type $II_1$ if and only if $B-tI$ is
	Hilbert--Schmidt where $0 < t < 1$; then $\pi_B$ is a factor
	representation of type $II_1$.
  \item[(ii)] $\,\,\psi_B$ is of type $II_\infty$  if and only if
	(a) $B(I-B)(B-pI)^2$ is trace class where $0 < t < 1$ and (b) the
	essential spectrum of $B$ contains $0$ or $1$; then $\pi_B$ is a factor
	representation of type $II_\infty$.
  \item[(iii)] $\,\,\psi_B$ is of type $III$ if and only if $B(I-B)(B-pI)^2$
	is not of trace class whenever $0 < t < 1$; then $\pi_B$ is a factor
        representation of type $III$.
\end{itemize}
\end{itemize}
Similar considerations hold for $SU(\infty)$, $SO(\infty)$ and $Sp(\infty)$.
This gives an indication of the delicacy in choice of 
type of representations of $M_\R$.  Clearly factor representations of
type $I$ and $II_1$ will be the easiest to deal with.
\smallskip

It is worthwhile to consider the case where the inducing representation
$\kappa\otimes e^{i\sigma}$ is trivial on $M_\R$, in other words is a
unitary character on $P_\R$. In the finite dimensional
case this leads to a $K_\R$--fixed vector, spherical functions on $G_\R$
and functions on the symmetric space $G_\R/K_\R$.  In the infinite
dimensional case it leads to open problems, but there are a few examples
(\cite{DOW}, \cite{W5}) that may give accurate indications.

\section{Parabolic Induction}\label{sec9}
\setcounter{equation}{0}

We view $\kappa\otimes e^{i\sigma}$ as a 
representation $man \mapsto e^{i\sigma}(a)\kappa(m)$ of $P_\R = M_\R A_\R N_\R$
on $E_\kappa$.  It is well defined because $N_\R$ is a closed normal subgroup 
of $P_\R$.  Let $\cU(\gg_\C)$ denote the universal enveloping algebra of 
$\gg_\C$.  The {\em algebraically induced representation} is given on the
Lie algebra level as the left multiplication action of $\gg_\C$ on
$\cU(\gg_\C) \otimes_{\gp_\R} E_\kappa$,
$$
d\pi_{\kappa,\sigma,alg}(\xi): \cU(\gg_\C) \otimes_{\gp_\R} E_\kappa \to
\cU(\gg_\C) \otimes_{\gp_\R} E_\kappa \text{ by } 
 \eta\otimes e \mapsto (\xi\eta)\otimes e.
$$
If $\xi \in \gp_\R$ then
$
d\pi_{\kappa,\sigma,alg}(\xi)(\eta\otimes e) = \Ad(\xi)\eta \otimes e + 
\eta\otimes d(\kappa\otimes e^{i\sigma})(\xi)e\, .
$
To obtain the associated representation $\pi_{\kappa,\sigma}$ of $G_\R$ we
need a $G_\R$--invariant completion of $\cU(\gg_\C) \otimes_{\gp_\R} E_\kappa$
so that the
$\pi_{\kappa,\sigma,alg}(\exp(\xi)) := \exp(d\pi_{\kappa,\sigma,alg}(\xi))$
are well defined.  For example we could use a $C^k$ completion,
$k \in \{0, 1, 2, \dots, \infty, \omega\}$, representation of $G_\R$ on $C^k$ 
sections of the vector bundle $\E_{\kappa\otimes e^{i\sigma}} \to G_\R/P_\R$
associated to the action $\kappa\otimes e^{i\sigma}$ of $P_\R$ on $E_\kappa$.
The representation space is
$$
\{\varphi: G_\R \to  E_\kappa \mid \varphi \text{ is } C^k \text{ and }
  \varphi(xman) = 
  e^{i\sigma}(a)^{-1}\kappa(m)^{-1}f(x)\} 
$$
where $m \in M_\R$\,, $a \in A_\R$ and $n \in N_\R$, and the action of
$G_\R$ is
$[\pi_{\kappa,\sigma,C^k}(x)(\varphi)](z) = \varphi(x^{-1}z)$.
In some cases one can unitarize $d\pi_{\kappa,\sigma,alg}$ by constructing a
Hilbert space of sections of $\E_{\kappa\otimes e^{i\sigma}} \to G_\R/P_\R$.
This has been worked out explicitly when $P_\R$ is a direct limit of minimal
parabolic subgroups of the $G_{n,\R}$ \cite{W5}, and more generally it comes
down to transitivity of $K_\R$ on $G_\R/P_\R$ \cite{W7}.  
In any case the resulting representations of $G_\R$ depend on the choice 
of class $\cC_{M_\R}$ of representations of $M_\R$.

\vskip .2in

Department of Mathematics, University of California,Berkeley, CA 94720--3840, USA

{\tt e-mail: jawolf@math.berkeley.edu}
\end{document}